\documentclass[a4paper,12pt]{amsart}
\usepackage[utf8]{inputenc}

\usepackage{mathrsfs}
\usepackage{graphicx}
\usepackage{hyperref}
\usepackage{enumerate}
\usepackage{multicol}
\usepackage{color}
\usepackage{amsmath,amssymb,amscd}
\usepackage{amsthm}
\usepackage{mathabx}

\usepackage{pdfpages}

\usepackage{times}
\usepackage[varg]{txfonts}

\usepackage[top=1in, bottom=1in, left=0.75in, right=0.75in]{geometry}

\theoremstyle{plain}
\newtheorem{thm}{Theorem}

\newtheorem{cor}{Corollary}
\newtheorem{prop}{Proposition}

\theoremstyle{definition}
\newtheorem{dfn}{Definition}

\theoremstyle{remark}
\newtheorem*{rem}{Remarks}

\newcommand{\C}{\mathbb{C}}
\newcommand{\N}{\mathbb{N}}

\title{Moment functions on groups}
\author{\.Zywilla Fechner, Eszter Gselmann and László Székelyhidi}

\address[\.Z.~Fechner]{Institute of Mathematics \\Łódź University of Technology \\ 90-924 Łódź, ul. Wólczańska 215 \\ Poland}
\email{zywilla.fechner@p.lodz.pl}

\address[E.~Gselmann]{Institute of Mathematics \\University of Debrecen \\H-4002 Debrecen, P.O.Box: 400 \\ Hungary}
\email{gselmann@science.unideb.hu}

\address[L.~Székelyhidi]{Institute of Mathematics \\University of Debrecen \\H-4002 Debrecen, P.O.Box: 400 \\ Hungary}
\email{szekely@science.unideb.hu, lszekelyhidi@gmail.com}

\begin{document}
%\nocite{*}

\begin{abstract}
 The main purpose of this work is to prove characterization theorems for generalized moment functions on groups. 
 According one of the main results these are exponential polynomials that can be described with the aid of 
 complete (exponential) Bell polynomials. These characterizations will be immediate consequences of our main result about the 
 characterization of generalized moment functions of higher rank. 
\end{abstract}


\maketitle

\section{Introduction}

Let $(G,+)$ be an Abelian group. Recall that a nonzero function $m\colon G\to \C$ is called {\it exponential}, if 
$$
m(x+y)=m(x)m(y)
$$ 
holds for all $x,y$ in $G$. 
Let $N$ be a nonnegative integer. A function $\varphi\colon G\to \C$ is termed to be a {\it moment function of order} $N$, if there exist functions $\varphi_k\colon G \to \C$ such that $\varphi_0=1$, $\varphi_N=\varphi$ and
\begin{equation}
\varphi_k(x+y)=\sum_{j=0}^k {k\choose j} \varphi_j(x)\varphi_{k-j}(y),
\label{eq:MomentFunctionsOnGroup}
\end{equation} 
for all $x$ and $y$ in $G$ ($k=0,1,\dots,N$). We generalize this concept by relaxing the assumption \hbox{$\varphi_0=1$} to $\varphi_0(0)=1$. 
In this case $\varphi_0$ is an arbitrary exponential function and we say that $\varphi_0$ {\it generates the generalized moment sequence of order} $N$ and the function $\varphi_k$ is a {\it generalized moment function of order} $k$, or, if we want to specify the exponential $\varphi_0$, then we say that $\varphi_k$ is a {\it generalized moment function of order} $k$ {\it associated with the exponential} $\varphi_0$. Problems about moment functions have been extensively studied on different type of abstract structures, in particular on hypergroups. 
They are measure algebras where convolution has specific probabilistic interpretation. 
In the hypergroup-setting the justification of the name `moment problems' is more visible than on abstract groups without additional structure  as we can interpret moment functions by means of properly defined moments of random elements. 
For a more detailed discussion see e.g.  \cite{FecSze19a} and \cite{Sze13}. 
More on moment functions can be found e.g. in \cite{FecSze19,OroSze04,OroSze05,OroSze08, Sze18, SzeVaj12a, SzeVaj12} and references therein. As already mentioned groups are special hypergroups and in this paper we are going to focus on a generalized moment problem on Abelian groups.

Equation \eqref{eq:MomentFunctionsOnGroup} is closely related to the well-known functions of binomial type. We are particular interested in 
\eqref{eq:MomentFunctionsOnGroup} defined on abstract structures. 
A detailed discussion about binomial type equations in abstract setting can be found in \cite{Acz77}. 
In addition, this paper of J.~Aczél has provided motivation for the present research. 
Namely, it is shown there that if $(G,+)$ is a grupoid and $R$ is a commutative ring, then functions $\varphi_n\colon G\to R$ satisfying \eqref{eq:MomentFunctionsOnGroup} for each $n$ in $\N$ are of the form
\begin{equation}
\varphi_n(t)=n!\sum_{j_1+2j_2+\dots+nj_n=n}\prod_{k=1}^n \frac{1}{j_k!}\left(\frac{a_k(t)}{k!}  \right)^{j_k},
\label{eq:AczelPolySol}
\end{equation}
for all $t$ in $G$ and $k$ in $\N$ and arbitrary homomorphisms $a_k$ from $(G,+)$ into $(R,+)$.

The present paper is organized as follows: first we give the definition and a multi-variable characterization of generalized moment functions. 
Next we introduce the notion of generalized moment functions of rank $r$, a generalization of the soultions of \eqref{eq:MomentFunctionsOnGroup}, which represent generalized moment functions of rank $1$.
Our main result is the description of generalized moment functions of higher rank by means of Bell polynomials. As a corollary we get 
a new characterization of generalized moment functions on Abelian groups, as well. 

We recall the definition of the multinomial coefficient which will be used below: let $n$ be a nonnegative integer and $l$a positive integer. Then we write
$$ 
\binom{n}{k_{1}, k_2,\ldots, k_{l}}=\frac{n!}{k_{1}! k_2!\cdots k_{l}!} 
$$
for all $k_{1}, k_2,\ldots, k_{l}$ in $\N$ satisfying $ k_{1}+k_2+ \ldots+ k_{l}=n$. Observe that for $l=2$ we obtain the binomial coefficient.

For the sake of completeness we clarify that in this paper the set of complex numbers is denoted by $\C$, the set of nonzero complex numbers by $\C^{\times}$ and the set of nonnegative integers by $\N$. 

\section{Generalized moment functions on groups}

We begin with a multi-variable characterization of moment functions. 

\begin{prop}
 Let $G$ be an Abelian group, $l$, $N$ positive integers with $l\geq 2$, and let
 \hbox{$\varphi_{0}, \varphi_1,\ldots, \varphi_{N}\colon G\to \C$} be functions for which 
 \begin{equation}\label{Eq2}
  \varphi_{n}(x_{1}+x_2+\cdots+x_{l})= \sum_{\substack{k_{1}, k_2,\ldots, k_{l}\geq 0 \\ k_{1}+k_2+\cdots +k_{l}=n}}\binom{n}{k_{1}, k_2,\ldots, k_{l}} \prod_{t=1}^{l}\varphi_{k_{t}}(x_{t})
 \end{equation}
holds for each $x_{1}, x_2,\ldots, x_{l}$ in $G$ and $n=0, 1,\ldots, N$. If $\varphi_{0}(0)=0$, then all the functions $\varphi_{0}, \varphi_1,\ldots, \varphi_{N}$ are identically zero. 
If $\varphi_{0}(0)=1$, then the functions $\varphi_{0}, \varphi_1,\ldots, \varphi_{N}\colon G\to \C$ form a generalized moment sequence of order $N$. 
\end{prop}

\begin{proof}
 Assume that conditions of the theorem are satisfied.
 If $l=2$, then equation \eqref{Eq2} reduces to \eqref{eq:MomentFunctionsOnGroup}. 
 
 Suppose that $l>2$ and let $x, y$ be in $G$. 
 With the substitution  $x_{1}= x, x_{2}= y $ and $x_{i}=0$  for $i=3,4,\dots,n$
we get from equation \eqref{Eq2} that 
\[
 \varphi_{n}(x+y)= \sum_{\substack{k_{1},k_2, \ldots, k_{l}\geq 0\\ k_{1}+k_2+\cdots+k_{l}=n}}
 \binom{n}{k_{1}, k_2,\ldots, k_{l}}\varphi_{k_{1}}(x)\varphi_{k_{2}}(y) \cdot \prod_{t=3}^{l}\varphi_{k_{t}}(0)
\]
for each $n=0, 1,\ldots, N$. Thus the only thing we should prove is that $\varphi_{n}(0)=0$ if $n\geq 1$. 
Indeed, in this case the right hand side of the above identity is nothing but $\displaystyle\sum_{k=0}^{n}\binom{n}{k} \varphi_{k}(x)\varphi_{n-k}(y)$. 

With $n=0$ and $x_{i}=0$ for $i=1, 2,\ldots, l$ equation \eqref{Eq2}  leads to 
\[
 \varphi_{0}(0)=\varphi_{0}(0)^{l}, 
\]
that is, 
\[
 \varphi_{0}(0) \cdot  \left(\varphi_{0}(0)^{l-1}-1\right)=0, 
\]
from which we derive that either $\varphi_{0}(0)=0$ or $\varphi_{0}(0)$ is an $(l-1)$\textsuperscript{th} root of unity. 

Furthermore, equation \eqref{Eq2} for $n=0$ and $x_{1}, x_2,\ldots, x_{l}$ in $G$ yields that 
\[
 \varphi_{0}(x_{1}+x_2+\cdots+x_{l})=\prod_{i=1}^{l}\varphi_{0}(x_{i}) 
 , 
\]
which means that there exists an exponential $m\colon G\to \C$ such that 
\[
 \varphi_{0}(x)= \varphi_{0}(0)\cdot m(x) 
\]
holds for each $x$ in $G$.

Similarly, equation \eqref{Eq2} for $n=1$ with $x_{1}=x$ in $G$ and $x_{i}=0$ for $i=2,3,\dots,n$ yields that 
\[
 \varphi_{1}(x)= \varphi_{1}(x)\varphi_{0}(0)^{l-1} +(l-1)\varphi_{0}(x)\varphi_{1}(0)\varphi_{0}(0)^{l-2}.  
\]
From this, it follows that 
\begin{enumerate}[(A)]
 \item either $\varphi_{0}(0)=0$, implying that both $\varphi_{0}$ and $\varphi_{1}$ are identically zero;
 \item or $\varphi_{0}(0)$ is an $(l-1)$\textsuperscript{th} root of unity and 
 \[
  (l-1)\varphi_{0}(x)\varphi_{1}(0)\varphi_{0}(0)^{l-2} =0
 \]
which happens only when $\varphi_{1}(0)=0$, otherwise $\varphi_{0}\equiv 0$ would follow contrary to the fact that 
$\varphi_{0}(0)$ is an $(l-1)$\textsuperscript{th} root of unity. 
\end{enumerate}

Assume now that there exists $k$ in $\left\{1, \ldots, N-1\right\}$ such that $\varphi_{0}(0)\neq 0$ and 
\[
 \varphi_{1}(0)=\varphi_{2}(0)= \ldots = \varphi_{k}(0)=0. 
\]
In this case equation \eqref{Eq2}, with $k+1$ instead of $n$ and with the substitution 
$x_{1}=x$ in $G$ and   $x_{i}=0$ for $i=2,3, \ldots, l$,
yields that 
\begin{multline*}
 \varphi_{k+1}(x)
 = \sum_{\substack{t_{1}, t_2,\ldots, t_{l}\geq 0\\ t_{1}+t_2+\cdots+t_{l}=k+1}} 
 \binom{k+1}{t_{1}, t_2,\ldots, t_{l}} \varphi_{t_{1}}(x)\varphi_{t_{2}}(0) \cdots \varphi_{t_{l}}(0)
 \\
 = \varphi_{k+1}(x)\varphi_{0}^{l-1}(0)+\varphi_{k}(x)\varphi_{1}(0)\varphi_{0}(0)^{l-2}+ \cdots+ \varphi_{0}(x)\varphi_{0}^{l-2}(0)\cdots\varphi_{k+1}(0). 
\end{multline*}

Again, we have the following two alternatives: 
\begin{enumerate}[(A)]
 \item either $\varphi_{0}(0)=0$, implying that $\varphi_{k+1}$ is identically zero; 
 \item or $\varphi_{0}$ is an $(l-1)$\textsuperscript{th} root of unity and 
 \[
  \varphi_{k}(x)\varphi_{1}(0)\varphi_{0}(0)^{l-2}+ \cdots+ \varphi_{0}(x)\varphi_{0}^{l-2}(0)\cdots\varphi_{k+1}(0)=0 
  \qquad 
  \left(x\in G\right). 
 \]
\end{enumerate}
Due to the induction hypothesis $\varphi_{1}(0)= \varphi_2(0)=\ldots = \varphi_{k}(0)=0$, thus 
\[
 \varphi_{0}(x)\varphi_{0}^{l-2}(0)\cdots\varphi_{k+1}(0)=0. 
\]
Since $\varphi_{0}(0)\neq 0$, $\varphi_{0}$ is not identically zero, which shows that the only possibility is that $\varphi_{k+1}(0)=0$. 

Finally we conclude that $\varphi_{1}(0)=\varphi_{2}(0)= \ldots= \varphi_{N}(0)=0$. However,  this implies for $x,y$ in $G$:
\[
 \sum_{\substack{k_{1}, k_2,\ldots, k_{l}\geq 0\\ k_{1}+k_2+\cdots+k_{l}=n}}
 \binom{n}{k_{1}, k_2,\ldots, k_{m}}\varphi_{k_{l}}(x)\varphi_{k_{2}}(x) \cdot \prod_{t=3}^{l}\varphi_{k_{t}}(0)
 =
 \sum_{k=0}^{n}\binom{n}{k} \varphi_{k}(x)\varphi_{n-k}(y). 
\]
Therefore, if $\varphi_{0}(0)=1$, then the functions $\varphi_{0}, \varphi_{1}, \ldots, \varphi_{N}$ form a generalized moment sequence of order $N$. 
\end{proof} 

\begin{prop}
  Let $G$ be an Abelian group, and let $l$, $N$ be positive integers with $l\geq 2$. If the functions
 \hbox{$\varphi_{0}, \varphi_1,\ldots, \varphi_{N}\colon G\to \C$}  form a moment sequence of order $N$, then the system of equations 
 \[
  \varphi_{n}(x_{1}+x_2+\cdots+x_{l})= \sum_{\substack{k_{1}, k_2,\ldots, k_{l}\geq 0 \\ k_{1}+k_2+\cdots +k_{l}=n}}\binom{n}{k_{1}, k_2,\ldots, k_{l}} \prod_{t=1}^{l}\varphi_{k_{t}}(x_{t})
 \]
is satisfied for each $x_{1}, x_2,\ldots, x_{l}$ in $G$ and $n=0, 1,\ldots, N$. 
\end{prop}

\begin{proof}
We prove the statement by induction on $l$. 
Assume that functions $\varphi_{0}, \varphi_{1}, \ldots, \varphi_{N}$ constitute a moment sequence of order $N$, i.e.
\[
 \varphi_{n}(x+y)= \sum_{k=0}^{n}\binom{n}{k} \varphi_{k}(x)\varphi_{n-k}(y)
\]
is satisfied for each $x, y\in G$. 
First observe that for $n=0$ the statement holds trivially, because both equations are the same, namely 
\[
\varphi_0(x+y)=\varphi_0(x)\varphi_0(y)
\qquad 
\left(x, y\in G\right).
\]
Thus it is enough to consider the case $n\geq 1$. 
Secondly, note that from \eqref{eq:MomentFunctionsOnGroup} we have $\varphi_0(0)=0$ or $\varphi_0(0)=1$ and for all $n\geq 1$ we have $\varphi_n(0)=0$. 

Let now $n$ in $\left\{1,\dots,N   \right\}$ be arbitrarily fixed and take $x_1, x_2$ in $G$. 
Applying \eqref{Eq2} for $m=2$ and $n$ we have
\[
\varphi_n(x_1+x_2)=\sum_{k=0}^n {n\choose k} \varphi_{k}(x_1)\varphi_{n-k}(x_2). 
\]
Observe that the right hand side of this equation can be written as 
\[
 \sum_{k=0}^n {n\choose k} \varphi_{k}(x_1)\varphi_{n-k}(x_2)= \sum_{k_1,k_2\geq 0, k_1+k_2=n} {n \choose k_1,k_2}\varphi_{k_1}(x_1)\varphi_{k_2}(x_2)
 \qquad 
 \left(x_{1}, x_{2}\in G\right), 
\]
whixch yields that the statement holds for $l=2$. 
Assume that there exists an integer $l\geq 2$  for which the statement holds. 
Then 
\[
 \varphi_{n}(x_1+\dots+(x_l+x_{l+1}))
\\
=\sum_{\substack{k_{1}, \ldots, k_{l-1},k_l+k_{l+1}\geq 0 \\ k_{1}+\cdots +k_{l-1}+k_l+k_{l+1}=n}}\binom{n}{k_{1}, \ldots, k_{l-1}, k_l+k_{l+1}} \prod_{t=1}^{l-1}\varphi_{k_{t}}(x_{t})\varphi_{k_l+k_{l+1}}(x_l+x_{l+1})
\]
for any $x_{1}, \ldots, x_{l}, x_{l+1}$ in $G$. 
If we use that the statement holds for $l=2$ and the induction hypothesis, then we get that  
\begin{multline*}
\sum_{\substack{k_{1}, \ldots, k_{l-1},k_l+k_{l+1}\geq 0 \\ k_{1}+\cdots +k_{l-1}+k_l+k_{l+1}=n}}\binom{n}{k_{1}, \ldots, k_{l-1}, k_l+k_{l+1}} \prod_{t=1}^{l-1}\varphi_{k_{t}}(x_{t})\varphi_{k_l+k_{l+1}}(x_l+x_{l+1})
\\=
\sum_{\substack{k_{1}, \ldots, k_{l-1},k_l+k_{l+1}\geq 0 \\ k_{1}+\cdots +k_{l-1}+k_l+k_{l+1}=n}}\binom{n}{k_{1}, \ldots, k_{l-1}, k_l+k_{l+1}} \prod_{t=1}^{l-1}\varphi_{k_{t}}(x_{t}) {k_l+k_{l+1}\choose k_l} \varphi_l(x_l) \varphi_{l+1}(x_{l+1})
\\
=
\sum_{\substack{k_{1}, k_2,\ldots, k_{l+1}\geq 0 \\ k_{1}+\cdots +k_{l}+k_{l+1}=n}}\binom{n}{k_{1}, k_2,\ldots, k_{l+1}} \prod_{t=1}^{l+1}\varphi_{k_{t}}(x_{t})
\qquad 
\left(x_{1}, \ldots, x_{l}, x_{l+1} \in G\right). 
\end{multline*}
Thus the statement holds for $l+1$, too. 
\end{proof}

\section{Generalized moment functions of rank $r$}

In this section we extend the notion of generalized moment functions. 
%To do this, the following notation and terminology is necessary concerning multi-indices. 
At the same time, while stating and proving our results about generalized moment functions of higher rank, notions and results from spectral analysis and synthesis are required. 

Let $G$ be a discrete Abelian group. Subsequently, if it is not otherwise stated, the group $G$ is always endowed with this topology. 

Given a function $f$ on this group and an element $y$ in $G$, we define
\[
 \Delta_{f;y}=\delta_{-y}-f(y)\delta_{0},  
\]
where $\delta_{-y}$ and $\delta_{0}$ denote the point masses concentrated at $-y$ and $0$, respectively. 
For the product of modified differences we use the notation 
\[
 \Delta_{f; y_{1}, y_2,\ldots, y_{n+1}}= \prod_{i=1}^{n+1}\Delta_{f; y_{i}}, 
\]
for any positive integer $n$ and for each $y_{1}, y_2,\ldots, y_{n+1}$ in $G$. 
Here the product $\prod$ on the right hand side is meant as a convolution product of the measures $\Delta_{f; y_{i}}$ for all $i=1, 2,\ldots, n+1$. 

For each function $f\colon G\to \mathbb{C}$ the ideal in $\mathbb{C}G$ generated by all modified differences of the form $\Delta_{f;y}$ with $y$ in $G$, is denoted by $M_{f}$. 
Due to Theorem 12.5 of \cite{Sze14}, this ideal $M_{f}$ is proper if and only if 
$f$ is an exponential and in this case  we call $M_{f}= \tau(f)^{\perp}$  an {\it exponential maximal ideal}. Recall that $\tau(f)$ denotes the variety of the function $f$ and it is nothing but the smallest closed translation invariant subspace of $\mathscr{C}(G)$ that contains the function $f$. 

Let $m$ be an exponential on the group $G$ and $n$  a positive integer. The function 
$f\colon G\to \mathbb{C}$ is a \emph{generalized exponential monomial} of degree at most 
$n$ corresponding to the exponential $m$ if 
\[
 \Delta_{m; y_{1}, y_2,\ldots, y_{n+1}}\ast f=0
\]
holds for each $y_{1}, y_2,\ldots, y_{n+1}$ in $G$. Equivalently, $f$ is a generalized exponential monomial if there exists an exponential $m$ and a positive integer $n$ such that its annihilator includes a positive power of an exponential maximal ideal, that is
$M_{m}^{n+1}\subset \tau(f)^{\perp}$ holds. 

A special class of exponential monomials is formed by those corresponding to the exponential identically $1$: these are called \emph{generalized polynomials}.

The function $f\colon G\to \mathbb{C}$ is called an \emph{exponential monomial} if it is a generalized exponential monomial and $\tau(f)$ is finite dimensional. Finite sums of exponential monomials are called \emph{exponential polynomials}. Exponential monomials corresponding to the 
exponential identically $1$ are called \emph{polynomials}. 

In connection with exponential polynomials here we also recall Theorem 12.31 from \cite{Sze14}. 

\begin{thm}
 Let $G$ be an Abelian group, and let 
 $f\colon G\to \mathbb{C}$ be an exponential polynomial. Then there exist positive integers $n, k$ and for each $i=1, 2,\ldots, n$ and $j=1, 2,\ldots, k$ there exists a polynomial $P_{i}\colon \mathbb{C}^{k}\to \mathbb{C}$, an exponential $m_{i}$ and a homomorphism $a_{j}$ of $G$ into the additive group of complex numbers such that 
 \[
  f(x)=\sum_{i=1}^{n}P_{i}\left(a_{1}(x), a_2(x),\ldots, a_{k}(x)\right)m_{i}(x)
 \]
for each $x$ in $G$. Conversely, every function of this form is an exponential polynomial. 
\end{thm}

A \emph{composition} of a nonnegative integer $n$ is a sequence of nonnegative integers $\alpha= \left(\alpha_{k}\right)_{k\in \N}$ such that 
\[
 n= \sum_{k=1}^{\infty}\alpha_{k}. 
\]
For a positive integer $r$, an \emph{$r$-composition} of a nonnegative integer $n$ is a composition 
$\alpha= \left(\alpha_{k}\right)_{k\in \N}$ with 
$\alpha_{k}=0$ for $k>r$. 

Given a sequence of variables $x=(x_{k})_{k\in \N}$ and compositions $\alpha= \left(\alpha_{k}\right)_{k\in \N}$ and 
$\beta= \left(\beta_{k}\right)_{k\in \N}$ we define 
\[
 \alpha!=\prod_{k=1}^{\infty}\alpha_{k},\quad \left| \alpha\right| = \sum_{k=1}^{\infty}\alpha_{k}, \quad 
 x^{\alpha}=\prod_{k=1}^{\infty}x_{k}^{\alpha_{k}},\quad \binom{\alpha}{\beta}= \prod_{k=1}^{\infty}\binom{\alpha_{k}}{\beta_{k}}.
\]
Furthermore, 
$\beta \leq \alpha$ means that $\beta_{k}\leq \alpha_{k}$ for all $k\in \N$ and 
$\beta < \alpha$ stands for $\beta \leq \alpha$  and $\beta \neq \alpha$.

\begin{dfn}
Let $G$ be an Abelian group, $r$ a positive integer, and for each multi-index $\alpha$ in $\N^r$ 
let $f_{\alpha}\colon G\to \C$ be continuous function. We say that $(f_{\alpha})_{\alpha \in \N^{r}}$ is a \emph{generalized 
moment sequence of rank $r$}, if 
\begin{equation}\label{Eq3}
f_{\alpha}(x+y)=\sum_{\beta\leq \alpha} \binom{\alpha}{\beta} f_{\beta}(x)f_{\alpha-\beta}(y)
\end{equation}
holds whenever $x,y$ are in $G$. The function $f_0$, where $0$ is the zero element in $\N^r$, is called the {\it generating function} of the sequence.
\end{dfn}

\begin{rem}
\begin{enumerate}[(i)]
\item\label{rem_i}  For $r=1$, instead of multi-indices, we have `ordinary' indices and \eqref{Eq3} is nothing but
 \[
  f_{\alpha}(x+y)= \sum_{\beta=0}^{\alpha}\binom{\alpha}{\beta}f_{\beta}(x)f_{\alpha-\beta}(y) 
  \qquad 
  \left(x, y\in G\right)
 \]
for each nonnegative integer $\alpha$, showing that generalized moment functions of rank $1$ are moment sequences. 
\item\label{rem_ii} For $\alpha=(0, \ldots,0)$ we have
$$
f_{0, \ldots,0}(x+y)=f_{0, \ldots,0}(x)\cdot f_{0, \ldots,0}(y) 
\qquad 
\left(x, y\in G\right)
$$
hence $f_{0, \ldots,0}=m$ is an exponential, or identicaly zero. In what follows, when considering generalized moment function sequences of any rank, we always assume that the generating function is not identically zero, hence it is always an expoential.
\item\label{rem_iii} Now let $\alpha$ be in $\N^{r}$ with $|\alpha|=1$. In this case we have for each positive integer $n$ that 
$$
f_{n\cdot \alpha}(x+y)=\sum_{k=1}^n \binom{n}{k} f_{k\cdot \alpha}(x)f_{(n-k)\cdot \alpha}(y) 
\qquad 
\left(x, y\in G\right). 
$$
Hence $(f_{n\cdot \alpha})_{n\in \N}$ is a generalized moment function sequence associated with the exponential \hbox{$m=f_{0\cdot \alpha}$. }

\item\label{rem_iv} Using the above definition  for $r=2$, for example for  $|\alpha|\leq 2$ we have 
\[
\begin{array}{rcl}
f_{0,0}(x+y)&=&f_{0,0}(x)f_{0,0}(y)\\
f_{1,0}(x+y)&=&f_{1,0}(x)f_{0,0}(y)+f_{1,0}(y)f_{0,0}(x)\\
f_{0,1}(x+y)&=&f_{0,1}(x)f_{0,0}(y)+f_{0,1}(y)f_{0,0}(x)\\
f_{2,0}(x+y)&=&f_{2,0}(x)f_{0,0}(y)+2f_{1,0}(x)f_{1,0}(y)+f_{2,0}(y)f_{0,0}(x)\\
f_{0,2}(x+y)&=&f_{0,2}(x)f_{0,0}(y)+2f_{0,1}(x)f_{0,1}(y)+f_{0,2}(y)f_{0,0}(x)\\
f_{1,1}(x+y)&=&f_{1,1}(x)f_{0,0}(y)+f_{1,0}(x)f_{0,1}(y)+f_{1,0}(y)f_{0,1}(x)+f_{1,1}(y)f_{0,0}(x)
\end{array}
\qquad 
\left(x, y\in G\right). 
\]

In view of the first equation (or using remark (\ref{rem_ii})), we immediately get that the function 
$f_{0, 0}=m$ is an exponential. Furthermore, due to the second and the third equation, the functions 
$f_{0, 1}$ and $f_{0, 1}$ are $m$-sine functions. In other words,  $\left\{f_{0, 0}, f_{0, 1}\right\}$ and
$\left\{f_{0, 0}, f_{1, 0}\right\}$ form a moment sequence of order $1$. 
More generally, due to remark (\ref{rem_iii}), $(f_{0, n})_{n\in \N}$ and 
also $(f_{n, 0})_{n\in \N}$ is a moment function sequence associated with the exponential $f_{0, 0}=m$.

\item\label{rem_v} Assume that $(f_{\alpha})_{\alpha\in \N^{2}}$ is a generalized moment sequence of rank two. 
For all nonnegative integer $n$, define the function $\varphi_{n}\colon G\to \C$ by 
\[
 \varphi_{n}(x)= \sum_{k=0}^{n}\binom{n}{k}f_{k, n-k}(x) 
 \qquad 
 \left(x\in G\right). 
\]
Then $\left(\varphi_{n}\right)_{n\in \N}$ is a moment function sequence associated with the exponential 
$\varphi_{0}=f_{0, 0}=m.$

Indeed, for $n=0$, 
\[
 \varphi_{0}(x)=f_{0, 0}(x)
 \qquad 
 \left(x\in X\right), 
\]
which is an exponential, as we wrote above. 

For each positive $n$, we have 
\begin{multline*}
 \sum_{k=0}^{n}\binom{n}{k}\varphi_{k}(x)\varphi_{n-k}(y)
 =
 \sum_{k=0}^{n}\binom{n}{k} \cdot \left[\sum_{i=0}^{k}\binom{k}{i}f_{i, k-i}(x)\right]\cdot \left[\sum_{j=0}^{n-k}\binom{n-k}{j}f_{j, n-k-j}(y)\right]
 \\
 =
 \sum_{k=0}^{n}\binom{n}{k} \left[\sum_{i=0}^{k} \sum_{j=0}^{n-k}\binom{k}{i} \binom{n-k}{j} f_{i, k-i}(x)f_{j, n-k-j}(y)\right]
 \\
 =
 \sum_{k=0}^{n}\binom{n}{k} f_{k, n-k}(x+y)= \varphi_{n}(x+y) 
 \qquad 
 \left(x, y\in G\right). 
\end{multline*}
\item\label{rem_vii} The sequence of functions $(f_{\alpha})_{\alpha \in \N^{r}}$ is a generalized moment sequence of rank $r$ associated 
with the nonzero exponential $f_{0, \ldots, 0}=m$ \emph{if and only if} 
$({f_{\alpha}}/{m})_{\alpha \in \N^{r}}$ is a generalized moment sequence of rank $r$ associated with the exponential which is identically one. 
\item\label{rem_viii} Let $(f_{\alpha})_{\alpha \in \N^{r}}$ be a generalized moment sequence of rank $r$ and denote $m= f_{0, \ldots, 0}$. 
Then for all multi-index $\alpha$ in $\N^{r}$, the function $f_{\alpha}$ is a generalized exponential monomial of degree at most 
$|\alpha|$ corresponding to the exponential $m$. Thus for any multi-index $\alpha$ in $\N^{r}$, there exists an a generalized 
polynomial $p_{\alpha}\colon G\to \C$ of degree at most $|\alpha|$ such that $f_{\alpha}= p_{\alpha} \cdot m$. 

If $|\alpha|=0$, that is, if $\alpha= (0, 0,\ldots, 0)$, then $f_{0, \ldots, 0}=m$ is an exponential. Thus the above statement holds trivially. 

In case $|\alpha|=1$, then  there exists $i\in \left\{1, 2,\ldots, r\right\}$ such that 
\[
 \alpha_{i}=1 \quad \text{and} \quad \alpha_{j}=0 \quad \text{for any } j\in \left\{1, \ldots, r\right\}, j\neq i. 
\]
This, in view of remark (\ref{rem_iii}) implies that $f_{\alpha}$ is an $m$-sine function, i.e. 
\[
 f_{\alpha}(x+y)= f_{\alpha}(x)m(y)+m(x)f_{\alpha}(y) 
 \qquad 
 \left(x, y\in G\right). 
\]
This implies however that 
\[
 \Delta_{m; y_{1}, y_{2}}f_{\alpha}(x)=0 
 \qquad 
 \left(x, y_{1}, y_{2}\in G\right), 
\]
in other words $f_{\alpha}$ is a generalized exponential monomial of degree at most one corresponding to the exponential $m$. 

Assume now that there exists a multi-index $\alpha$ in $\N^{r}$ such that the statement holds for any multi-index $\beta$ for which 
$\beta < \alpha$. Since we have 
\[
  f_{\alpha}(x+y)= \sum_{\beta\leq \alpha }\binom{\alpha}{\beta}f_{\beta}(x)f_{\alpha-\beta}(y) 
\]
for each $x, y\in G$, 
\[
\Delta_{m; y}f_{\alpha}(x) =  \sum_{\beta <\alpha }\binom{\alpha}{\beta}f_{\beta}(x)f_{\alpha-\beta}(y) 
\qquad 
\left(x, y\in G\right). 
\]
Observe that the right hand side of this equation (as a functions of the variable $x$) is, due to the induction hypothesis, a generalized exponential polynomial of degree at most 
$|\alpha|-1$, showing that $f_{\alpha}$ is a generalized exponential monomial of degree at most $|\alpha|$. 

\item\label{rem_ix} The previous remark can be strengthened. Namely, generalized moment sequences of rank $r$ not only generalized exponential monomials, but they are 
generalized exponential monomials. 
This means that the variety of any such function is finite dimensional. 
Indeed, let $(f_{\alpha})_{\alpha \in \N^{r}}$ be a generalized moment sequence of rank $r$. 
Then for any multi-index $\alpha$ in  $\N^{r}$ we have 
\[
 f_{\alpha}(x+y)= \sum_{\beta \leq \alpha}\binom{\alpha}{\beta}f_{\beta}(x)f_{\alpha-\beta}(y) 
 \qquad 
 \left(x, y\in G\right), 
\]
showing that 
\[
 \tau(f)\subseteq \mathrm{span}\left\{f_{\beta}\, \vert \, \beta \leq \alpha\right\}. 
\]
The linear space on the right hand side is obviously finite dimensional, implying the same for $\tau(f)$. 

We conclude that the polynomial $p_{\alpha}$, appearing in remark \eqref{rem_viii}, is a complex polynomial of some complex-valued additive functions defined on $G$. 

We can summarize that, for any multi-index $\alpha$ in $\N^{r}$ there exists a complex polynomial $P$, and for each $\beta\leq \alpha$ there is an 
additive function $a_{\beta}\colon G\to \C$ such that 
\[
 f_{\alpha}(x)= P(a(x))m(x) 
 \qquad 
 \left(x\in G\right). 
\]
Here $a_{\beta}$ is the $\beta$\textsuperscript{th} coordinate function of $a$. 
\item\label{rem_x} If $k$ and $1\leq n_{1}< n_{2} < \cdots < n_{k} \leq r$ are positive integers, then let $\pi_{n_{1}, \ldots, n_{k}}\alpha$ denote that element in 
$\N^{r}$ which differs from $\alpha$ in only those coordinates that do not belong to the set $\left\{n_{1}, \ldots, n_{k}\right\}$, 
and those coordinates are zero.  

With this notation, let $k$ and $r$ be positive integers with $k\leq r$ and  $1\leq n_{1}< n_{2} < \cdots < n_{k} \leq r$ be also positive integers. 
If $(f_{\alpha})_{\alpha \in \N^{r}}$ is a generalized moment sequence of rank $r$ on the group $G$, then the sequence of functions 
$(f_{\pi_{n_{1}, \ldots, n_{k}}\alpha})_{\alpha \in \N^{r}}$ form a generalized moment sequence of rank $k$. 
\end{enumerate}
\end{rem}

As we will see, while describing the polynomial $P$ appearing in remark \eqref{rem_ix}, a well-known sequence of polynomials and an appropriate addition 
formula will play a distinguished role. 
Let us consider the sequence of complex polynomials $\left(B_{n}\right)_{n\in \N}$ defined through the following recurrence: for each $t,t_1,t_2,\dots,t_{n+1}$ in $\C$ we let 
\[
 \begin{array}{rcll}
  B_{0}(t)&=&1, &\\
  B_{n+1}(t_{1}, \ldots, t_{n+1})&=&\displaystyle\sum_{i=0}^{n}\binom{n}{i}B_{n-i}(t_{1}, \ldots, t_{n-i})t_{i+1} &\left(n=0,1,\dots\right).
 \end{array}
\]
Alternatively, we can also use the double series expansion of the generating function 
\[
 \exp\left(\sum_{j=1}^{\infty}x_{j}\frac{t^{j}}{j!}\right)= \sum_{n=0}^{\infty}B_{n}(x_{1}, \ldots, x_{n})\frac{t^{n}}{n!}. 
\]
We call $B_{n}$  the \emph{$n$\textsuperscript{th} complete (exponential) Bell polynomial}. 

For these functions the following addition formula holds:
\[
 B_{n}(t_{1}+u_{1}, \ldots, t_{n}+u_{n})=
 B_{n}(t_{1}, \ldots, t_{n})+\sum_{k=1}^{n-1}\binom{n}{k} B_{n-k}(t_{1}, \ldots, t_{n-k})B_{k}(u_{1}, \ldots, u_{k}) 
 +B_{n}(u_{1}, \ldots, u_{n}) 
\]
for any positive integer $n$, and for each complex numbers $t_{1}, \ldots, t_{n}$ and $u_{1}, \ldots, u_{n}$. 

The first few complete (exponential) Bell polynomials are
\[
\begin{array}{rl}
    B_{0}={}&1\\[8pt]
    B_{1}(x_{1})={}&x_{1}\\[8pt]
    B_{2}(x_{1},x_{2})={}&x_{1}^{2}+x_{2}\\[8pt]
    B_{3}(x_{1},x_{2},x_{3})={}&x_{1}^{3}+3x_{1}x_{2}+x_{3}\\[8pt]
    B_{4}(x_{1},x_{2},x_{3},x_{4})={}&x_{1}^{4}+6x_{1}^{2}x_{2}+4x_{1}x_{3}+3x_{2}^{2}+x_{4}\\[8pt]
    B_{5}(x_{1},x_{2},x_{3},x_{4},x_{5})={}&x_{1}^{5}+10x_{2}x_{1}^{3}+15x_{2}^{2}x_{1}+10x_{3}x_{1}^{2}+10x_{3}x_{2}+5x_{4}x_{1}+x_{5}\\[8pt]
    B_{6}(x_{1},x_{2},x_{3},x_{4},x_{5},x_{6})={}&x_{1}^{6}+15x_{2}x_{1}^{4}+20x_{3}x_{1}^{3}+45x_{2}^{2}x_{1}^{2}+15x_{2}^{3}+60x_{3}x_{2}x_{1}\\&{}+15x_{4}x_{1}^{2}+10x_{3}^{2}+15x_{4}x_{2}+6x_{5}x_{1}+x_{6}\\[8pt]
    B_{7}(x_{1},x_{2},x_{3},x_{4},x_{5},x_{6},x_{7})={}&x_{1}^{7}+21x_{1}^{5}x_{2}+35x_{1}^{4}x_{3}+105x_{1}^{3}x_{2}^{2}+35x_{1}^{3}x_{4}\\&{}+210x_{1}^{2}x_{2}x_{3}+105x_{1}x_{2}^{3}+21x_{1}^{2}x_{5}+105x_{1}x_{2}x_{4}\\&{}+70x_{1}x_{3}^{2}+105x_{2}^{2}x_{3}+7x_{1}x_{6}+21x_{2}x_{5}+35x_{3}x_{4}+x_{7}.\
    \end{array}
\]

It turns out, however, that these functions are useful in the rank one case, only. To cover the general case as well, a multivariate extension of the Bell polynomials is necessary. 
Here we follow the notation and the terminology of Port \cite{Por94}. 

The \emph{multivariate Bell polynomials} are the functions $ B_{\alpha}$ defined by
\[
 B_{\alpha}(x)= \alpha! \sum \prod_{\mu}\frac{(x_{\mu}^{c_{\mu}})}{c_{\mu}! (\mu!)^{c_{\mu}}}, 
\]
where the sum is taken over $\alpha= \sum_{0< \beta< \alpha} c_{\mu}\mu$. 

Similarly as above, they can also be introduced using the double series expansion of the generating function 
\[
 \sum_{0\leq |\alpha|}B_{\alpha}(x)\frac{t^{\alpha}}{\alpha!}= \exp \left(\sum_{1\leq |\mu|} x_{\mu}\frac{t^{\mu}}{\mu!}\right). 
\]

Let now $\alpha= (\alpha_{k})_{k\in \N}$ be a composition and 
$x=(x_{k})_{k\in \N}$ and $y=(y_{k})_{k\in \N}$ be sequences of variables. Then 
\[
 B_{\alpha}(x+y)= \sum_{\beta\leq \alpha}\binom{\alpha}{\beta}B_{\beta}(x)B_{\alpha-\beta}(y). 
\]

For instance, if $r=2$, then the first few multivariate Bell polynomials are the following: 
\[
 \begin{array}{rcl}
  B_{0, 0}(x)&=&1\\[8pt]
  B_{0, 1}(x_{0, 1})&=& x_{0, 1}\\[8pt]
  B_{1, 0}(x_{1, 0})&=& x_{1, 0}\\[8pt]
  B_{1, 1}(x_{0, 1}, x_{1, 0}, x_{1, 1})&=& x_{0, 1} x_{1, 0}+x_{1, 1}\\[8pt]
  B_{2, 0}(x_{0, 1}, x_{1, 0}, x_{1, 1}, x_{2, 0})&=& x_{1, 0}^{2}+x_{2, 0}\\[8pt]
  B_{0, 2}(x_{0, 1}, x_{1, 0}, x_{1, 1}, x_{2, 0})&=& x_{0, 1}^{2}+x_{0, 2}\\[8pt]
  B_{2, 1}(x_{0, 1}, x_{1, 0}, x_{1, 1}, x_{2, 0}, x_{2, 1})&=& x_{0, 1}x_{1, 0}^{2}+2x_{1, 0}x_{1, 1}+x_{0, 1}x_{2, 0}+x_{2, 1}\\[8pt]
  B_{2, 2}(x_{0, 1}, x_{1, 0}, x_{1, 1}, x_{2, 0}, x_{1, 2}, x_{2, 1}, x_{2, 2})&=& x_{0, 1}^{2}x_{1, 0}^{2}+x_{0, 2}x_{1, 0}^{2}+4x_{0, 1}x_{1, 0}x_{1, 1}+2x_{1, 1}^{2}+2x_{1, 0}x_{1, 2}\\
  &&+x_{0, 1}^{2}x_{2, 0}+x_{0, 2}x_{2, 0}+2x_{0, 1}x_{2, 1}+x_{2, 2}
 \end{array}
\]

It is important to point out that if $r=1$ then the multivariate Bell polynomials reduce to the complete (exponential) Bell polynomials. 

\begin{prop}
 Let $G$ be a commutative group, $r$ a positive integer,
$m\colon G\to \C$ an exponential and let
 $a= \left(a_{\alpha}\right)_{\alpha\in \N^{r}}$ be a sequence of complex-valued additive functions defined on $G$.  We define the functions $f_{\alpha}\colon G\to \C$  by 
 \[
  f_{\alpha}(x)= B_{\alpha}(a(x))m(x) 
  \qquad 
  \left(x\in G\right). 
 \]
Then $(f_{\alpha})_{\alpha\in \N^{r}}$ forms a generalized moment sequence of rank $r$ associated with the exponential $m$. 
\end{prop}

\begin{proof}
 If $|\alpha|=0$, then the above formula is 
 \[
  f_{0, \ldots, 0}(x)=m(x) 
  \qquad 
  \left(x\in G\right), 
 \]
that is, $f_{0, \ldots, 0}$ is an exponential. 

Let now $\alpha$ be in $\N^{r}$ with $|\alpha|>0$, then 
\begin{multline*}
f_{\alpha}(x+y)
= 
B_{\alpha}(a(x+y))m(x+y)
\\
=
B_{\alpha}(a(x)+a(y))m(x)m(y) 
=
\sum_{\beta}\binom{\alpha}{\beta}B_{\beta}(a(x))B_{\alpha-\beta}(a(y))m(x)m(y)
\\=
\sum_{\beta}\binom{\alpha}{\beta}B_{\beta}(a(x))m(x) \cdot B_{\alpha-\beta}(a(y))m(y)
=
\sum_{\beta}\binom{\alpha}{\beta}f_{\beta}(x)f_{\alpha-\beta}(y)
\end{multline*}
for each $x, y$ in $G$. 
\end{proof}

The following result is about the converse of the previous statement. 

\begin{thm}
 Let $G$ be a commutative group, $r$ a positive integer, and for each $\alpha$ in $\N^{r}$, 
 let $f_{\alpha}\colon G\to \C$ be a function. If the sequence of functions 
 $(f_{\alpha})_{\alpha\in \N^{r}}$ forms a generalized moment sequence of rank $r$, then there exists an 
 exponential $m\colon G\to \C$ and a sequence of complex-valued additive functions $a= (a_{\alpha})_{\alpha\in \N^{r}}$ 
 such that  for every multi-index $\alpha$ in $\alpha$ in $\N^{r}$ and $x$ in $G$ we have 
 \[
  f_{\alpha}(x)=B_{\alpha}(a(x))m(x). 
 \]
\end{thm}

\begin{proof}
Let $r$ be a fixed positive integer, $G$ be a commutative group, let furthermore for each $\alpha\in \N^{r}$, 
 $f_{\alpha}\colon G\to \C$ be a function. 
 Assume that the sequence of functions  $(f_{\alpha})_{\alpha\in \N^{r}}$ forms a generalized moment sequence of rank $r$. 
 We prove the statement by induction on the height of the multi-index $\alpha$. 
 
 If $|\alpha|=0$, then we necessarily have $\alpha= (0, 0,\ldots, 0)$ and 
 \[
  f_{0, 0,\ldots, 0}(x+y)= f_{0, 0,\ldots, 0}(x)f_{0, 0,\ldots, 0}(y) 
  \qquad 
  \left(x, y\in G\right), 
 \]
which gives immediately that there exists an exponential $m\colon G\to \C$ such that 
$ f_{0, 0,\ldots, 0}= m$.
Furthermore, if $|\alpha|=1$, then there exists $i\in \left\{1, 2,\ldots, r\right\}$ such that 
$\alpha_{i}=1$ and $\alpha_{j}=0$ for each $j$ in $\{1, 2,\ldots, r\}$, $j\neq i$. 
In view of remark \eqref{rem_iii} this implies that $f_{\alpha}$ is an $m$-sine function, that is
\[
 f_{\alpha}(x)= B_{\alpha}(a(x))m(x) 
\]
holds for $x$ in $G$.

Now we assume that there exists a multi-index $\alpha$ in $\N^{r}$ such that the statement holds true for everymulti-index $\beta$ with
$\beta < \alpha$. This means that for any multi-index $\beta<\alpha$ there exists an additive function $a_{\beta}\colon G\to \C$ 
such that the representation in the statement of the theorem holds. We have to show how the additive function $a_{\alpha}\colon X\to \C$ should be constructed. 
Since $f_{\alpha}$ is a generalized moment function of order $\alpha$ and of rank $r$, we have 
\begin{multline*}
 f_{\alpha}(x+y)-f_{\alpha}(x)m(y)-m(x)f_{\alpha}(y)= 
 \sum_{0<\beta < \alpha}\binom{\alpha}{\beta} f_{\beta}(x)f_{\alpha-\beta}(y)
 \\
 =
 \sum_{0<\beta < \alpha}\binom{\alpha}{\beta} B_{\beta}(a(x))m(x)B_{\alpha-\beta}(a(y))m(y) 
 \quad 
 \left(x, y\in G\right). 
\end{multline*}
Let now $\chi \colon G\to \mathbb{C}$ be an arbitrary additive function. 
Then, due to the addition formula of the $|\alpha|$-variable fuction $B_{\alpha}$, we have that 
\begin{multline*}
 B_{\alpha}(a(x+y), \chi(x+y))m(x+y)-B_{\alpha}(a(x), \chi(x))m(x)\cdot m(y)-B_{\alpha}(a(y), \chi(y))m(y)\cdot m(x)
 \\
 =
 \sum_{0<\beta < \alpha}\binom{\alpha}{\beta} B_{\beta}(a(x))m(x)B_{\alpha-\beta}(a(y))m(y)
\end{multline*}
holds for each $x, y$ in $G$. 
Observe, that this means that the function $S_{\alpha}\colon G\to \mathbb{C}$ defined by 
\[
 S_{\alpha}(x)= f_{\alpha}(x)-B_{\alpha}(a(x), \chi(x)) 
 \qquad 
 \left(x\in G\right)
\]
is an $m$-sine function. Thus, there exists an additive function $\eta \colon G\to \mathbb{C}$ such that 
$S(x)= \eta(x)m(x)$ for all $x$ in $G$, hence 
\[
 f_{\alpha}(x)= B_{\alpha}(a(x), \chi(x))m(x)+\eta(x)m(x)
 = 
 B_{\alpha}(a(x), \chi(x)+\eta(x))m(x) 
 \qquad 
 \left(x\in G\right), 
\]
since $B_{\alpha}$ is additive in its last variable. 

We conclude that there exists an additive function $\xi \colon G\to \mathbb{C}$ such that 
\[
 f_{\alpha}(x)=B_{\alpha}(a(x), \xi(x))\cdot m(x)
\]
for each $x$ in $G$ and the theorem is proved.
\end{proof}

Observe that the additive function $\xi$ appearing in the above formula is uniquely determined, that is it does not depend on the choice of 
the function $\chi$ appearing at the first part of the proof. Indeed, if 
\[
 f_{\alpha}(x)=B_{\alpha}(a(x), \xi_{1}(x))\cdot m(x) 
 \quad 
 \text{and}
 \quad 
 f_{\alpha}(x)=B_{\alpha}(a(x), \xi_{2}(x))\cdot m(x)
\]
holds for each $x$ in $G$ with the additive functions $\xi_{1}, \xi_{2}\colon G\to \mathbb{C}$, 
then 
\[
 B_{\alpha}(a(x), \xi_{1}(x)) = B_{\alpha}(a(x), \xi_{2}(x))
\]
follows for ech $x$ in $G$. Recall that the function $B_{\alpha}$ is additive in its last variable and 
\[
 B_{\alpha}(a(x), \xi_{i}(x))= P_{\alpha}(a(x))+\xi_{i}(x) 
 \qquad 
 \left(x\in G, i=1, 2\right)
\]
holds with a certain polynomial $P_{\alpha}$ which does not depend neither on $\xi_{1}$ nor on $\xi_{2}$ which implies 
\[
 \xi_{1}(x)=\xi_{2}(x)
\]
for all $x$ in $G$. 

As a consequence of this result, the characterization of generalized moment functions (of rank one) follows. Namely, functions of the form $(B_{n}\circ a)\cdot m$ are generalized moment functions. More precisely, we have the following. 

\begin{cor}
 Let $G$ be a commutative group, $n$ an arbitrary positive integer, 
 $m\colon G\to \C$ an exponential, and let
 $a_{1}, a_2,\ldots, a_{n}\colon G\to \C$ be additive functions. 
We define the sequence of functions $(f_{n})_{n\in \N}$ for each $x$ in $G$ by
 \[
  f_{n}(x)= B_{n}(a_{1}(x), \ldots, a_{n}(x))m(x). 
 \]
Then $(f_{n})_{n\in \N }$ forms a generalized moment sequence associated with the exponential $m$. 
\end{cor}

On the other hand, we can deduce the following result:

\begin{cor}
 Let $G$ be a commutative group, and for each natural number $n$ let 
 $f_{n}\colon G\to \C$ be a function. If the sequence 
 $(f_{n})_{n\in \N }$ forms a generalized moment sequence, then there exists an 
 exponential $m\colon G\to \C$ and there are additive functions 
 $a_{1}, a_2,\ldots, a_{n}\colon G\to \C$ such that 
 \[
  f_{n}(x)=B_{n}\left(a_{1}(x), \ldots, a_{n}(x)\right)m(x) 
 \]
 holds for every $x$ in $G$ and $n=1,2,\dots$.
\end{cor}

%\bibliographystyle{plain}
%\bibliography{moment}

\end{document}